\documentclass{amsart}
\usepackage{amsfonts}
\usepackage{amscd}
\usepackage{amssymb}
\newtheorem{theorem}{Theorem}
\theoremstyle{plain}

\newtheorem{lemma}[theorem]{Lemma}

\theoremstyle{definition}

\newcommand{\R}{{\mathbb R}}
\newcommand{\Z}{{\mathbb Z}}
\newcommand{\C}{{\mathbb C}}
\newcommand{\N}{{\mathbb N}}
\newcommand{\X}{{\mathbb X}}

\newcommand{\Ca}{{\mathbb O}}
\newcommand{\Ha}{{\mathbb H}}
\newcommand{\Fi}{{\mathbb F}}

\newcommand{\Sch}{{\mathcal C^2}}
\newcommand{\A}{{\mathcal A}}
\newcommand{\F}{{\mathcal F}}
\renewcommand{\S}{{\mathcal S}}
\renewcommand{\Re}{{\mathop{\mathrm{Re}}\,}}
\renewcommand{\Im}{{\mathop{\mathrm{Im}}\,}}

\def\Sp{\mathrm{Sp}}
\def\U{\mathrm{U}}

\def\F{\mathrm{F}}

\def\sh{\sinh}

\def\e-s{e^{-s}}

\def\de-s{e^{-2s}}

\usepackage[usenames]{color}

\normalbaselines
\begin{document}

\title[Asymptotics for the Radon transform]
{Asymptotics for the Radon transform\\on hyperbolic spaces}
\author[Andersen and Flensted--Jensen]
{Nils Byrial Andersen and Mogens Flensted--Jensen}
\address{Department of Mathematics,
Aarhus University,
Ny Munkegade 118,
Building 1530,
DK-8000 Aarhus C,
Denmark}
\email{byrial@imf.au.dk}
\address{Department of Mathematical Sciences,
University of Copenhagen, Universitetsparken 5,
DK-2100 Copenhagen {\O}, Denmark}
\email{mfj@life.ku.dk}
\begin{abstract}
Let $G/H$ be a hyperbolic space over $\R, \, \C$ or $\Ha$, and
let $K$ be a maximal compact subgroup of $G$. Let $D$ denote a certain explicit
invariant differential operator, such that the non-cuspidal discrete series belong to the kernel of $D$.
For any $L^2$-Schwartz function $f$ on $G/H$, we prove that
the Abel transform $\A(Df)$ of $Df$ is a Schwartz function. This is an extension of a result
established in \cite{AF-J} for $K$-finite and $K\cap H$-invariant functions.
\end{abstract}
\maketitle

\section{Introduction}

The {\it Radon transform} $R$ on the hyperbolic spaces $G/H$,
\[
R f=  \int_{N^*} f(\cdot nH)\,dn,
\]
where $N^*\subset G$ is a certain unipotent subgroup,
and the associated {\it Abel transform} $\A$,
were introduced and studied in \cite{AF-JS} and \cite{AF-J}.
Generalizing Harish-Chandra's notion of cusp forms for real
semisimple Lie groups, a discrete series is said to be {\it cuspidal} if it is annihilated
by the Radon transform. In contrast with the Lie group case, however,
{\it non-cuspidal} discrete series exist. For the projective hyperbolic spaces, these are precisely
the spherical discrete series, but for some real non-projective hyperbolic spaces, there also exist non-spherical non-cuspidal
discrete series.

Let $\Sch(G/H)$ denote the space of $L^2$-Schwartz functions on $G/H$.
Except for some boundary cases, $\A$ maps $\Sch(G/H)$ into Schwartz functions
in the absense of non-cuspidal discrete series.
On the other hand, $\A f$ can be explicitly calculated
for functions $f$ belonging to the non-cuspidal discrete series.
To complete the picture, we prove below that $\A$ essentially maps
the orthocomplement in $\Sch(G/H)$ of the non-cuspidal discrete series into Schwartz functions.
To be more precise, let
$\Delta_\rho =\Delta + \rho_{\mathfrak q}^2$, where $\Delta$ denotes the Laplace--Beltrami operator on $G/H$,
and consider the $G$-invariant differential operator
$D = \Delta_\rho (\Delta_\rho - \lambda_1^2) \dots (\Delta_\rho - \lambda_r^2)$,
where $\lambda_1, \dots, \lambda_r$ are the parameters of the non-cuspidal discrete series.
Then $\A(Df)$ is a Schwartz function.
This extends our previous result, \cite[Theorem~6.1~(vi)]{AF-J}, valid only for the dense $G$-invariant subspace of $\Sch(G/H)$
generated by the $K$-irreducible $(K\cap H)$-invariant functions, to all Schwartz functions.

In \cite{AF-J} we also considered the exceptional case corresponding to the Cayley numbers $\Ca$.
We expect our new result to hold for this case as well,
but we have not been through the rather cumbersome details.

The second author wants to thank Professor Vladimir Molchanov for the invitation to visit
Tambov University, where our results were first reported, in November 2012.
We would also like to thank Henrik Schlichtkrull and Job Kuit for helpful discussions and comments.

\section{The Radon transform}

In this section, we define the Radon transform  and the Abel transform for the projective
hyperbolic spaces over the classical fields $\Fi = \R, \C$ and $\Ha$.
We have tried to keep the presentation and notation to a minimum, see \cite{AF-JS} and \cite{AF-J} for further details
(including results and proofs).

Let $x\mapsto \overline{x}$ be the standard (anti-) involution of $\mathbb F$.
Let $p\ge 0,\,q \ge 1$ be two integers, and consider the Hermitian form
$[ \cdot ,\cdot ]$ on $\mathbb F ^{p+q+2}$ given by
\[
[ x,y ]  =x_1 \overline{y}_1 + \dots + x_{p+1} \overline{y}_{p+1} - x_{p+2} \overline{y}_{p+2} -
\dots - x_{p+1+q+1} \overline{y}_{p+1+q+1},\,\, (x,y \in \mathbb F ^{p+q+2}).
\]
Let $G = \U (p+1,q+1;\mathbb F)$ denote the group of $(p+q+2)\times
(p+q+2)$ matrices over $\mathbb F$ preserving $[ \cdot ,\cdot ]$.
Thus $\U(p+1,q+1;\R)= \mathrm{O}(p+1,q+1)$,
$\U(p+1,q+1;\C)= \U(p+1,q+1)$ and $\U(p+1,q+1;\mathbb H)= \Sp(p+1,q+1)$
in standard notation.
Put $\U (p;\mathbb F) = \U (p,0;\mathbb F)$, and let $K=\U(p+1;\mathbb F)\times \U(q+1;\mathbb F)$
be the maximal compact subgroup of $G$ fixed by the Cartan involution on $G$.

Let $x_0 = (0,\dots,0,1)^T$, where superscript $T$ indicates transpose.
Let $H=\U (p+1,q;\mathbb F)\times \U(1;\mathbb F)$ be the subgroup of $G$ stabilizing the line $\mathbb F\cdot x_0$
in $\mathbb F^{p+q+2}$.
The reductive symmetric space $G/H$ can be identified with
the projective hyperbolic space $\X = \X (p+1,q+1;\mathbb F)$,
\[
\X= \{ z\in \mathbb F^{p+q+2} : [z,z] = -1\}/\sim,
\]
where $\sim$ is the equivalence relation $z\sim zu,\,u\in \mathbb F^*$.

Let $X_t$, for $t\in \R$, denote the element in the Lie algebra $\mathfrak g$ of $G$
with value $t$ in the $(1,p+q+2)$'th and $(p+q+2,1)$'th matrix entries (the two opposite corners in the anti-diagonal), and $0$ otherwise.
Let $\mathfrak a_{\mathfrak q}$ denote the Abelian subalgebra given by $\{X_t \mid t\in \R \}$,
let $a_{t} = \exp (X_{t})$ denote the exponential of $X_{t}$, and
also define $A_{\mathfrak q}=\exp (\mathfrak a_{\mathfrak q})$.

Let (considered as row vectors)
$$
u=(u_1,\dots,u_{p})\in \mathbb F^{p}\quad\text{and}\quad
v = (v_q, \dots ,v_1)\in\mathbb F^q,
$$
and let $w\in \Im \mathbb F$ (i.e.,\,$w = 0$ for $\mathbb F = \R$).
Define $N_{u,v,w}\in \mathfrak g$ as the matrix given by
\[
N_{u,v,w}=\left(
\begin{array}
[c]{cccc}%
-w & u & v & w\\
-\overline u^T & 0  & 0 & \overline u^T\\
\overline{v}^T & 0 & 0 & -\overline{v}^T \\
-w & u & v & w
\end{array}
\right).
\]
Then $\exp (N_{u,v,w}) =I+ N_{u,v,w}+ 1/2 N_{u,v,w} ^2$, and a small calculation yields that
\begin{equation}\label{aexpN}
a_{t}\exp ( N_{u,v,w}) \cdot x_0=
(\sinh t+ 1/2 e^{t}(|u|^2 - |v|^2)+e^{t}w, \overline u^T; -\overline{v}^T, \cosh t + 1/2 e^{t} (|u|^2 - |v|^2) +e^{t}w)^T,
\end{equation}
for any $t \in \R$.

Define the nilpotent subalgebra $\mathfrak n^*$ as follows,
for $p\ge q$,
\begin{equation} \label{nforpgeq}
\mathfrak n^* = \{N_{u,v,w} : u =
(-\overline{v^r},u'),\,v\in \mathbb F ^{q},\,u'\in \mathbb F^{p-q} \},
\end{equation}
and, for $p<q$,
\begin{equation}\label{nforplq}
\mathfrak n^* = \{N_{u,v,w} : v =
(-\overline{u^r},v'),\,u\in \mathbb F ^{p},\,v' \in \mathbb F^{q-p} \},
\end{equation}
where $u^r,v^r$ means that the order of the indices is reversed. By abuse of notation, we leave out the superscript 
${}^r$ in what follows.

We finally also define the following $\rho$-factors. Let $d=\dim _\R \F$, and let
$\rho _{\mathfrak q} = \frac 12(dp+dq + 2(d-1))\in \R$, and
$\rho _{1} = \frac 12((|dp-dq|+2(d-1))\in \R$.

Let $N^* = \exp(\mathfrak n^*)$ denote
the nilpotent subgroup generated by $\mathfrak n^*$.
For functions $f$ on $G/H$, we define, assuming convergence,
\begin{equation}
\label{defradon}
Rf(g) = \int _{N^*} f(g n^*H)\, dn^*  \qquad (g\in G).
\end{equation}

Let $f\in \Sch(G/H)$, the space of $L^2$-Schwartz functions on $G/H$.
From \cite{AF-JS} and \cite{AF-J}, we know that the Radon transform $Rf$ is a smooth
function. Also, the integral defining $R$ converges uniformly on compact sets, and
$R$ is $G$- and $\mathfrak g$-equivariant.

We define the associated Abel transform $\A$ by $\A f(a)=a^{\rho_1} Rf(a)$, for $a\in A_{\mathfrak q}$.
We are mainly interested in the values of $R f$ and $\A f$ on the elements $a_s$, and thus define
$Rf(s) = Rf(a_s)$, and, similarly, $\A f(s)= \A f (a_s)$, for $s\in \R$.
Let $\Delta$ denote the Laplace--Beltrami operator on $G/H$.
Then, for $f\in \Sch(G/H)$,
\begin{equation}\label{exchange}
\A (\Delta f) =\left (\frac{d^2}{ds^2}- \rho_{\mathfrak q} ^2\right )\A f \qquad (s\in \R).
\end{equation}

Finally, for $R>0$, let $C_R^\infty (G/H)$ denote the subspace of smooth functions on $G/H$
with support in the ($K$-invariant) `ball' $\{ka_s\cdot x_0\mid |s| \le R\}$ of radius $R$.
Similarly, let $C_R^\infty (\R)$ denote the subspace of smooth functions on $\R$
with support in $[-R,R]$, and let $\S (\R)$ denote the Schwartz space on $\R$.

\section{The discrete series and the Abel transform}

Let $q>1$, or $d>1$.
The discrete series for the projective hyperbolic spaces can then be parametrized as
\begin{equation*}
\{T_\lambda \,|\, \lambda = \frac 12 (dq-dp)-1+\mu _\lambda >0, \mu_\lambda \in 2\Z\},
\end{equation*}
see \cite{AF-JS} and \cite{AF-J}.
The spherical discrete series are given by the parameters $\lambda$ for which $\mu_\lambda\le 0$, including
the 'exceptional' discrete series corresponding to $\lambda>0$ for which $\mu_\lambda<0$.

For $q=d=1$, the discrete series is parameterized by $\lambda \in \R\backslash\{0\}$ such that
$|\lambda|+\rho_{\mathfrak q}\in 2\Z$, and there are no spherical discrete series.

The parameters $\lambda$ are, via the formula
$\Delta f=(\lambda^2-\rho_{\mathfrak q} ^2)f$, related to the eigenvalues of 
$\Delta$ acting on functions $f$ in the
corresponding representation space in $L^2(G/H)$.

We have a complete classification of the cuspidal and non-cuspidal discrete series for the projective hyperbolic spaces, also including
information about the asymptotics of the Radon and Abel transforms:
\begin{theorem}\label{Main-thm}
Let $G/H$ be a projective hyperbolic space over $\R, \, \C, \, \Ha$, with $p\ge 0, \, q\ge 1$.
\begin{enumerate}
\item[(i)] If $d(q-p) \le 2$, then all discrete series are cuspidal.
\item[(ii)]If $d(q-p) > 2,$ then non-cuspidal discrete series exists,
given by the parameters $\lambda > 0$ with $\mu_\lambda \le 0$.
More precisely, if $0 \ne f\in \Sch(G/H)$ belongs to $T_\lambda$, then $\A f(s)=C e^{\lambda s}$, with $C \ne 0$.
\item[(iii)] $T_\lambda$ is non-cuspidal if and only if $T_\lambda$ is spherical.
\item[(iv)] If $p\ge q$, and $f\in C^\infty _R (G/H)$, for $R>0$, then $\A f \in C_R^\infty (\R)$.
\item[(v)] If $d(q-p) \le 1$, and $f\in \Sch(G/H)$, then $\A f \in \S (\R)$.
\item[(vi)] Assume $d(q-p) > 1$. Let $D$ be the $G$-invariant differential operator
$\Delta_\rho (\Delta_\rho - \lambda_1^2) \dots (\Delta_\rho - \lambda_r^2)$,
where $\lambda_1, \dots, \lambda_r$ are the parameters of the non-cuspidal discrete series,
and $\Delta_\rho = \Delta + \rho_{\mathfrak q}^2$. Then $\A(Df) \in \S (\R)$, for $f\in \Sch(G/H)$.
\end{enumerate}
\end{theorem}

The above theorem is almost identical to \cite[Theorem~6.1]{AF-J}, except for item (vi), which was only proved for 
functions in the (dense) $G$-invariant subspace $\mathcal V$ of $\Sch(G/H)$ generated
by the $K$-irreducible $(K\cap H)$-invariant functions. Additionally, \cite[Theorem~6.1]{AF-J} furthermore included
the exceptional case corresponding to the Cayley numbers $\Ca$.

Theorem~\ref{Main-thm} (including the reformulation of (vi))
also holds for the real non-projective spaces $SO(p+1,q+1)_e / SO(p+1,q)_e$, except for item (iii),
due to the existence of non-cuspidal non-spherical discrete series corresponding to negative and odd values 
of $\mu_\lambda$ in the exceptional series, see \cite[Section~5]{AF-JS}.

The conditions in (vi) essentially state that $\A f$ is a Schwartz function if $f$ is perpendicular to all non-cuspidal discrete series.
The factor $\Delta_\rho$, however, seems to be necessary (except in the real case with $q-p$ odd),
even for the case $d(q-p)=2$, where there are no non-cuspidal discrete series.

In the next section, we prove Theorem~\ref{Main-thm}(vi).

\section{Proof of Theorem~\ref{Main-thm}(vi)}

First we note, following \cite[Section~10]{AF-J}, that the Schwartz decay conditions are satisfied
near $-\infty$ for $\A(f)$, and thus also for $\A(Df)$.
This leaves us to study the Abel transform near $+\infty$.

Let $f\in \Sch(G/H)$, and write $f[x]=f(gH)$, where $x=g\cdot x_0$. From (\ref{aexpN}) and  (\ref{nforplq}), we get
\begin{align*}
&\qquad \qquad\qquad Rf(s)=\int _{N^*} f (a_s n^*H) dn^* =\\
& \int _{\R ^{dq-dp} \times \R^{dp}\times \R^{d-1}}
f[(\sinh s - 1/2 e^s |v'|^2 +e^s w,u;-u,-v',\cosh s-1/2 e^s |v'|^2+e^s w)]\,dv'\,du\,dw.
\end{align*}
Let $v'=|v'|\overline v$, $v=-\sinh s + 1/2 e^s|v'|^2$, such that $|v'|^2= 1+2e^{-s}v-e^{-2s}$, and $\overline w=e^sw$.
Then,
\begin{align*}
&Rf(s)=e^{-ds}\int_{-\sinh s}^\infty \int _{\mathbb S^{dq-dp-1} \times \R^{dp}\times \R^{d-1}}
f[(\overline w-v,u;-u,-(1+2e^{-s}v-e^{-2s})^{1/2}\overline v,e^{-s}-v+\overline w)]\\
& \qquad \qquad\qquad \qquad \times(1+2e^{-s}v-e^{-2s})^{(dq-dp)/2-1} \, dv\, d\overline v\,du\, d\overline w,
\end{align*}
where $\mathbb S^r$ is the unit sphere in $\R^r$.

We will use the identification of $\X = \X (p+1,q+1;\mathbb F)$ with
\[
\X= \{ z\in \mathbb F^{p+q+2} : [z,z] < 0\}/\sim,
\]
and identify a function $f$ on $\X$ with a homogeneous function
on $\{z\in \mathbb F^{p+q+2} : [z,z] < 0\}$ of degree zero.

We now identify $\mathbb F^{p+q+2}$ with $\R^{d(p+q+2)}$ such that the coordinates satisfy 
$\Re z_{j}=x_{dj}$, for $j=1,\dots, p+q+2$.
Consider the real hyperbolic space
$\widetilde \X =\{z\in \mathbb F^{p+q+2} : [z,z] = -1\}$.
The group $\widetilde G= \mathrm O(d(p+1),d(q+1))$ acts transitively on $\widetilde \X$.
Let $\tilde K = \mathrm O(d(p+1))\times \mathrm O(d(q+1))$ denote the standard maximal compact subgroup of $\widetilde G$.
Let $U(\widetilde{\mathfrak k})$, respectively $U({\mathfrak k})$, denote the universal enveloping algebra of the
Lie algebra $\widetilde{\mathfrak k}$ of $\widetilde K$, respectively of the
Lie algebra ${\mathfrak k}$ of $K$.
\begin{lemma}\label{real case}
Let $U\in U(\tilde {\mathfrak k})$, then
$U$ maps $\Sch(G/H)$ into itself.
\end {lemma}
\begin{proof}
The lemma is obvious for $d=1$. So assume $d> 1$. 
We note that any element $x\in \widetilde \X$ can be written as $x=k a \cdot x_0$,
where $k \in K$, and $a=a_s, \, s\ge 0$.
Let $\widetilde H=\mathrm O(d(p+1),d(q+1)-1)$, and let $\tilde{\mathfrak m}$
denote the commutator of $A_{\mathfrak q}$ in the Lie algebra of $\widetilde K \cap \widetilde H$.
Then $\tilde {\mathfrak k} = \mathfrak k + \tilde{\mathfrak m}$.

Let $U_k= Ad(k)U$, for $k\in K$, then $Uf=(Ad(k^{-1})U_k)f$.
By the Campbell--Baker--Hausdorff formula, there exists an element $U^0_k \in U(\mathfrak k)$,
such that $U_k = U^0_k$ modulo the left ideal generated by $\tilde{\mathfrak m}$.
This implies that $Uf[ka\cdot x_0]=(Ad(k^{-1})U_{k}^0)f[ka\cdot x_0]$.
The map $k \mapsto Ad(k^{-1})U_k^0$ is continuous into a finite dimensional
subspace of $U({\mathfrak k})$, and we can write $Uf[ka\cdot x_0]=(Ad(k^{-1})U_{k}^0)f[ka\cdot x_0]=
\Sigma_i u_i(k)U_i f[ka\cdot x_0]$, for a finite set of elements $U_i \in U(\mathfrak k)$ and continuous coefficients $u_i(k)$.
It follows that $Uf$ is in $\Sch(G/H)$.
\end{proof}

Define for $t=(t_1, t_2, t_3) \in \R^3$, the auxiliary function
\begin{align*}
&G_f(t_1,t_2,t_3)= \int _{\mathbb S^{dq-dp-1} \times \R^{dp}\times \R^{d-1}}
f[(\overline w+t_1,u;-u,t_2\overline v,t_3+\overline w)] \,d\overline v\, du\, d\overline w,
\end{align*}
and, with the identification $z=e^{-s}$, define the function $F(z)=e^{ds}Rf(s)$. Then, since $\sinh s= -(z-z^{-1})/2$, we get
\begin{equation}\label{F_int}
F(z)=\int_{(z-z^{-1})/2}^\infty G_f(-v, -(1+2zv-z^2)^{1/2}, z-v)(1+2zv-z^2)^{(dq-dp)/2-1} \, dv.
\end{equation}

\begin{lemma}\label{function G}
The function $G_f$ is homogeneous of degree $dp+d-1$, for $t \in \{t \mid t_1^2-t_2^2-t_3^2<0 \}$,
$G_f$ is even in $t_2$, and satisfies $G_f(-t_1, t_2,-t_3)=G_f (t_1, t_2,t_3)$.

Let $X$ be the differential operator on $\R^3$ given by $t_3\partial / \partial t_2-t_2\partial/\partial t_3$.
For all $f\in \Sch(G/H)$, and all $k,N \in \N$, there exists a constant $C$, such that
\[
|X^k G_f(t)|\le C(t_2^2+t_3^2)^{-d(q-p)/4}(1+\log(t_2^2+t_3^2))^{-N},
\]
for all $t \in \{t \mid  t_1^2-t_2^2-t_3^2=-1 \}$.
\end {lemma}
\begin{proof}
The first statement follows from the homogeniety of $f$ and the definition of $G_f$.

As before we identify $\mathbb F^{p+q+2}$ with $\R^{d(p+q+2)}$. 
Define, for $i=d(1+2p)+1, \dots , d(1+p+q)$, the differential operator
$D_i f[x]= x_{d(p+q+2)}\partial / \partial x_i f[x]- x_i \partial / \partial x_{d(p+q+2)}f[x]$.
This operator is defined by the left action of an element $T_i$ in $\mathrm O (d(q+1))$ (with
value $1$ in the last entry of the i'th row, value $-1$ in the last entry of the i'th coloum, and $0$ otherwise),
and Lemma \ref{real case} thus gives that $D_i$ maps $\Sch(G/H)$ into itself.

Let now $\overline v = (v_{d(1+2p)+1}, \dots, v_{d(1+p+q)}) \in \mathbb S^{d(q-p)-1}$.
The operator
\[
Y_{\overline v}=\Sigma_{i=2+2p}^{1+p+q} v_i D_i,
\]
also maps $\Sch(G/H)$ into itself, and
\[
|Y_{\overline v}f[x]|\le d(q-p)\max_i (|D_i f[x]|).
\]
Applying the operator $X$ to the integrand in the definition of $G_f$, we get
\begin{align*}
Xf[t_1,u;-u,t_2 \overline v,t_3]&= t_3(\Sigma_{d(1+2p)+1}^{d(1+p+q)} \partial / \partial x_i f[.] v_i - t_2
\partial / \partial x_{d(p+q+2)}f[.]\\
&= t_3(\Sigma_{d(1+2p)+1}^{d(1+p+q)} \partial / \partial x_i f[.] v_i) - t_2 (\Sigma_{d(1+2p)+1}^{d(1+p+q)} v_i^2)
\partial / \partial x_{d(p+q+2)}f[.]\\
&= Y_{\overline v}f[t_1,u;-u,t_2 \overline v,t_3].
\end{align*}
The inequality for $X^k f$ follows from repeated use of this formula and from the asymptotic estimates of functions 
in $\Sch(G/H)$.
\end{proof}

In particular, it follows that the function $v\mapsto X^kG_f(-v, -1, -v)$ has the same parity as $k$.

\begin{lemma}\label{function dG}
Let $k_0$ be the largest integer such that $k_0 < (dq-dp)/2$, and
let $\epsilon =(dq-dp)/2 - k_0$.
Define $t=t(z,v)=(-v, -(1+2zv-z^2)^{1/2}, z-v)$. Then
\begin{itemize}
\item[(i)] For $k\le k_0$, the function
$v \mapsto \partial^k/\partial z^k(G_f(t(z,v))|(1+2zv-z^2)|^{(dq-dp)/2-1})$
is uniformly integrable over $\R$ for $z<1$.
\item[(ii)] For $k \le k_0$ odd, the function $v \mapsto \partial^k/\partial z^k(G_f(t(z,v))(1+2zv-z^2)^{(dq-dp)/2-1})$
is an odd function of $v$ for $z=0$.
\end{itemize}
\end{lemma}
\begin{proof}
Notice that $t_1^2-t_2^2-t_3^2=-1$ and $t_2^2+t_3^2=1+v^2$, for
$t=t(z,v)$, and that the integral (\ref{F_int}) is uniformly convergent for $0\le z \le K < \infty$.
The same holds with $G_f$ replaced by $X^kG_f$.

Repeated use of the formula
$\partial/\partial z(G_f(t(z,v))(1+2zv-z^2)^\alpha)=-XG_f(t(z,v))(1+2zv-z^2)^{\alpha-1/2}+2\alpha G_f(t(z,v))(1+2zv-z^2)^{\alpha-1}(z-v)$ yields (i), and together with the parity properties of $X^k G_f$ also gives (ii).
\end{proof}

We notice that $\epsilon =1$ if $d(q-p)$ is even, and $\epsilon= 1/2$ if $d(q-p)$ is odd, i.e.,\,if $d=1$ and $q-p$ is odd.

For $k < k_0$, the derivatives $\partial^k/\partial z^k$ of $G_f(t(z,v))(1+2zv-z^2)^{(dq-dp)/2-1}$
are zero at $v=-\sh s= \frac 12(z-z^{-1})$, whence the integrand is at least
$k_0$ times differentiable near $z=0$, and we can compute the
derivatives $d^k/d z^k F(z)$ by differentiating under the integral sign in (\ref{F_int}).

If $k_0>0$, we can use Taylors formula
to express $F(z)$ as a polynomial of degree $k_0-1$, plus a remainder term
involving $d^{k_0}/dz^{k_0}F(\xi)$, for some $0< \xi(z)< z$,
\[
F(z)= c_0 + c_1 z + c_2 z^2 + ... + c_{k_0-1} z^{k_0-1} + R_{k_0}(\xi) z^{k_0},
\]
where $0 < \xi < z $, and
\begin{align*}
 c_j =
&\frac 1{j!}\int_{-\infty}^\infty
{\frac{d^j}{dz^j}}{\bigg|_{z=0}}(G_f(t(z,v))(1+2zv-z^2)^{(dq-dp)/2-1}) \, dv,
\end{align*}
for $ j\in \{0, \dots, k_0-1\}$.
The remainder term is given by:
\begin{align*}
R_{k_0}(\xi)=
&\frac 1{k_0!}\int_{(\xi-\xi^{-1})/2}^\infty
\frac{d^{k_0}}{dz^{k_0}}\bigg|_{z=\xi} (G_f(t(z,v))(1+2zv-z^2)^{(dq-dp)/2-1}) \, dv.
\end{align*}

Consider $\A f(s)=e^{\rho_1 s} Rf(s) = z^{-(\rho_1 -d)} F(z)$, which is equal to
\begin{equation*}
c_0 z^{-(\rho_1 -d)} + c_1 z^{-(\rho_1 -d-1)} + c_2 z^{-(\rho_1 -d-2)} + ... + c_{k_0-1} z^{-\epsilon} +z^{(-\epsilon + 1)} R_{k_0}(\xi).
\end{equation*}
Here we have used that $\rho_1 -d= d(q-p)/2-1$.
For $j$ even, the exponents $-d(q-p)/2-1 -j$, for $j\in \{0, \dots, k_0-1\}$, correspond
to the parameters $\lambda_1, \dots, \lambda_r$ for the non-cuspidal discrete series, and $c_j=0$ for
$j$ odd, since the integrand is an odd function.

For the real non-projective hyperbolic spaces the condition concerning the parity $j$ does not hold, but in that case
{\it all} the exponents $-d(q-p)/2-1 -j$, for $j\in \{0, \dots, k_0-1\}$, correspond
to parameters $\lambda_1, \dots, \lambda_r$ for the non-cuspidal discrete series, see \cite[Section~3]{AF-JS}

From the definition of the differential operator $D$ and (\ref{exchange}),
we see that $\A (D f)$ at most has a contribution from the remainder term, 
and further that $\A (D f)$ does not have a constant term at $\infty$,
due to the term $d^2/d s^2$.
If $\epsilon=1/2$, the remainder term $e^{-1/2 s} R_{k_0}(\xi(s))$ is clearly rapidly decreasing, and we are thus
left to consider the case $\epsilon =1$, in which case $k_0=d(q-p)/2-1$.

Consider the constant term $C_{R_{k_0}}\ =\lim_{s \to \infty} R_{k_0}(\e-s)$, which could be non-zero.
We want to show that $R_{k_0}(\xi) - C_{R_{k_0}}$ is rapidly decreasing at $+\infty$,
where $\xi = \xi(s)$, with $0 < \xi < \e-s$.
We also include the case $k_0=0$, where we put $\xi=\e-s$.

Define $H(z,v)= \frac{d^{k_0}}{dz^{k_0}} (G_f(t(z,v))(1+2zv-z^2)^{k_0})$.
Then, for $\xi < z < 1$,
\[R_{k_0}(\xi) - C_{R_{k_0}}= \int_{(\xi-\xi^{-1})/2}^\infty (H(\xi,v)-H(0,v))\, dv +
\int_{-\infty}^{(\xi-\xi^{-1})/2} H(0,v)\, dv\, =\, I_1(\xi) + I_2(\xi).
\]
For $I_1(\xi)$, there exists $\xi_1=\xi_1(\xi,v) < \xi$, such that
$H(\xi,v)-H(0,v)= \xi d/dz_{|z=\xi_1} H(z,v)$, and we get:
\[
I_1(\xi) <  z \int_{-\infty}^\infty \left | \frac{d}{dz} \bigg|_{z=\xi_1} H(z,v) \right |\, dv.
\]
By Lemma~\ref{function dG}, the integrand is uniformly integrable for $z<1$,
and we conclude that $I_1(\xi)$ is bounded by $C\e-s$.

For $s$ large, the function $H(0,v)$ is for every $N \in \N$ bounded by
\[
|H(0,v)|\le C(1+v^2)^{-(-d(q-p)/4} |v|^{k_0}\log(1+v^2)^{-N},
\]
for some positive constant $C$. Using this, we find that
\[
I_2(z) < C \int^\infty_{\sh s} v^{-1} (\log (v))^{-N}\, dv = C(N-1)^{-1}(\log(\sh s))^{-N+1} \le C s^{-N+1}.
\]
It follows that $R_{k_0}(\xi) - C_{R_{k_0}}$ is rapidly decreasing at $+\infty$, whence $\A(Df)$
is rapidly decreasing at $+\infty$, which finishes the proof of Theorem~\ref{Main-thm}.


\begin{thebibliography}{99}

\bibitem{AF-JS}
N.B.~Andersen, M.~Flensted-Jensen and H.~Schlichtkrull,
{Cuspidal discrete series for semisimple symmetric spaces},
\textit{J.\ Funct.\ Anal.\ }{\bf 263} (2012), 2384--2408.

\bibitem{AF-J}
N.B.~Andersen, M.~Flensted-Jensen,
{Cuspidal discrete series for projective hyperbolic spaces},
Accepted for publication in \textit{Contemp.\ Math.\ }(arXiv:1209.3124).

\end{thebibliography}
\end{document}